\documentclass[twoside,10pt]{article}

\usepackage[pagewise]{lineno}

\usepackage[colorlinks,
            citecolor=red]{hyperref}

\usepackage{makecell}
\usepackage{multirow}

\usepackage[english]{babel}
\usepackage{amsmath}
\usepackage{amsthm}
\usepackage{amsfonts}
\usepackage{amssymb}
\usepackage{graphicx}
\usepackage{colortbl,dcolumn}
\usepackage{paralist}  
\usepackage{latexsym}
\usepackage{amsfonts}
\usepackage{amssymb}
\usepackage{cite}
\usepackage{appendix}
\usepackage{subcaption}
\usepackage{tikz}

     \newtheorem{lemma}{\bf Lemma}[section]
     \newtheorem{theorem}{\bf Theorem}[section]

     \newtheorem{definition}{\bf Definition}[section]
     \newtheorem{remark}{\bf Remark}[section]

     \numberwithin{equation}{section}

\begin{document}
	
\title{{\LARGE Effective Boundary Conditions for the Fisher-KPP Equation on a Domain with 3-dimensional Optimally Aligned Coating }
\footnotetext{E-mail addresses: gengxingri@u.nus.edu.\\ }}

\author{{Xingri Geng$^{a, b}$}\\[2mm]
	\small\it $^a$ Department of Mathematics, Southern University of Science and Technology, Shenzhen, P.R. China\\
	\small\it $^b$ Department of Mathematics, National University of Singapore,  Singapore 
}

\date{}
\maketitle

\begin{abstract}

We consider the Fisher-KPP equation on a three-dimensional domain surrounded by a thin layer whose diffusion rates are drastically different from that in the bulk. The bulk is isotropic, while the layer is considered to be anisotropic and  ``optimally aligned", where the normal direction is always an eigenvector of the diffusion tensor. To see the effect of the layer, we derive effective boundary conditions (EBCs) by the limiting solution of the Fisher-KPP equation as the thickness of the layer shrinks to zero. These EBCs contain some exotic boundary conditions including the Dirichlet-to-Neumann mapping and the Fractional Laplacian. Moreover, we
emphasize that each EBC keeps effective indefinitely, even as time approaches infinity.

\end{abstract}

\medskip


\medskip


\section{Introduction}

Motivated by the scenario of a nature reserve surrounded by a road/buffer zone. The diffusion rate is isotropic in the nature reserve, while it is anisotropic and drastically different in the road/buffer zone. Moreover, the road/buffer zone is thin compared to the scale of the nature reserve, resulting in multi-scales in the spatial variable. The difference in diffusion rate and spatial size leads to computational difficulty. To deal with such situations, we view the road/buffer zone as a thickless surface as its thickness shrinks to zero, on which ``effective boundary conditions" (EBCs) are imposed. These EBCs not only provide an alternative way for numerical computation but also give us an analytic interpretation of the effects of the road/buffer zone.

In this paper, we consider the Fisher-KPP equation on a 3-dimensional domain surrounded by a thin coating (see Figure \ref{fig}), which is anisotropic and ``optimally aligned" (a notion proposed in \cite{RW}) in the sense that any vector normal to the interface is always an eigenvector of the diffusion tensor. When the tangent diffusion rates in the coating are not larger than $o(1/\delta)$, EBCs and their lifespan were studied by Li, Wang, and Wu \cite{LWW2014}. This paper treats the case when the coating is ``optimally aligned" with arbitrary diffusion rates in the tangent direction, leading to the emergence of exotic EBCs. Moreover, we also show that the maximal time interval (called \textit{lifespan}) of each EBC during which the EBC remains effective is infinite.

\begin{center}
\begin{tikzpicture}
\def\angle{60}%
\pgfmathsetlengthmacro{\xoff}{2cm*cos(\angle)}%
\pgfmathsetlengthmacro{\yoff}{1cm*sin(\angle)}%
\draw (0,0) circle[x radius=3.5cm, y radius=2cm] ++(1.7*\xoff,1.7*\yoff) node{$\Omega_2$} ++(1.5*\xoff,-0.2*\yoff)  node{$\partial\Omega$};
\draw (0,0) circle[x radius=3cm, y radius=1.5cm] node{$\Omega_1$}
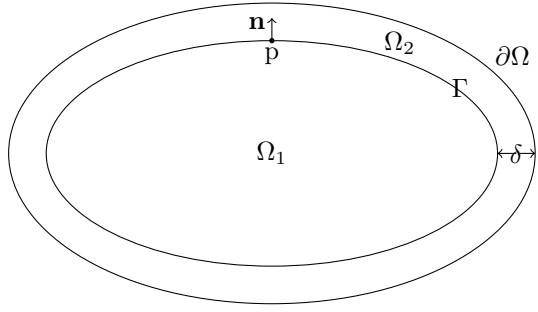
++(0*\xoff,-3*\yoff) node{Figure 1: $\Omega=\overline{\Omega}_1\cup\Omega_2$. $ \Omega_2$ is a thin layer with thickness $\delta$.} ++(2.5*\xoff,4*\yoff)  node{$\Gamma$};

\fill (0,1.5) circle [radius = 1pt];
\draw ++(0*\xoff,1.5*\yoff) node{p};

\draw[<-] (3,0)--(3.25,0) node{$\delta$};
\draw[->] (3.3,0)--(3.5,0);
\draw[->] (0,1.5)--(0,1.8);
\draw ++(-0.2*\xoff,2*\yoff) node{\textbf{n}};
\end{tikzpicture}\label{fig}
\end{center}

Let us introduce our mathematical model as follows: let the domain $\Omega_1$ be surrounded by the coating $\Omega_2$ with uniform thickness $\delta > 0$; let the domain $\Omega=\overline{\Omega}_1\cup\Omega_2\subset \mathbb{R}^3$ as shown in Figure \ref{fig}.  For any finite $T>0$,  consider the initial boundary value problem with the Dirichlet boundary condition
\begin{equation} \label{fkpp: PDE}
    \left\{
             \begin{array}{llr}
             u_t-\nabla \cdot (A(x)\nabla u)=f(u), &\mbox{$(x,t)\in Q_T,$}\\
             u=0,     &\mbox{$(x,t)\in S_T,$}&  \\
             u=u_0, &\mbox{$(x,t)\in\Omega\times\{0\},$}   
             \end{array}
  \right.
\end{equation}
where $Q_T:=\Omega\times(0,T)$,   $S_T:=\partial\Omega\times(0,T)$, $0 \leq u_0 \in L^\infty(\Omega)$, and $f(u) = u(1-u)$. Assume that the diffusion tensor $A(x)$ is given by
\begin{equation}\label{A}
    A(x)=\left\{
        \begin{aligned}
              &kI_{3\times3}, &x\in \Omega_1,\\
              &\left(a_{ij}(x)\right)_{3\times3}, & x\in\Omega_2,
        \end{aligned}
    \right.
\end{equation}
where $k$ is a positive constant independent of $\delta> 0$ and the positive-definite matrix $ (a_{ij} (x))$ is anisotropic and ``optimally aligned'' in the coating $\Omega_2$, and further satisfies
\begin{equation}\label{OAC}
     A(x)\textbf{n}(p)=\sigma \textbf{n}(p)  \; \text{ and } \; A(x)\textbf{s}(p)=\mu \textbf{s}(p), \quad \forall x\in \Omega_2, 
\end{equation}
where $\Omega_2$ is thin enough and $\Gamma$ is smooth enough such that the projection $p$ of $x$ onto $\Gamma$ is unique; $\textbf{n}(p)$ is the unit outer normal vector of $\Gamma$ at $p$; $\textbf{s}(p)$ is an arbitrary unit tangent vector of $\Gamma$ at $p$; $\sigma$ and $\mu$ are called the normal conductivity and the tangent conductivity, respectively.

\smallskip

One main purpose of this paper is to address EBCs on $\Gamma \times (0, T)$ as the thickness of the layer shrinks to zero. 
\begin{theorem}\label{fkpp: thm1}
Suppose that $A(x)$ is given in \eqref{A} and \eqref{OAC}. Let $0 \le u_0 \in L^\infty(\Omega)$ and $f(u) = u(1 - u)$. Assume further that $\sigma$ and $\mu$ satisfy
\begin{equation*}
    \lim_{\delta\to 0}\sigma\mu=\gamma\in[0,\infty], \quad
\lim_{\delta\to 0}\frac{\sigma}{\delta}=\alpha\in[0,\infty].
\end{equation*}
Let $u$ be the weak solution of (\ref{fkpp: PDE}), 
then as $\delta \to 0$, $u\to v$ weakly in $W_2^{1,0}(\Omega_1 \times (0,T))$ and strongly in $C([0,T];L^2(\Omega_1))$, where $v$ is the weak solution of 
\begin{equation} \label{fkpp: EPDE}
    \left\{
             \begin{array}{ll}
             v_t- k\Delta v=f(v), & (x,t)\in \Omega_1\times(0,T),  \\  
             v = u_0, & (x, t) \in \Omega_1\times \{0\},
             \end{array}
\right.
\end{equation}
subject to the effective boundary conditions on $\Gamma\times (0,T)$ listed in Table \ref{fkpp: tb}.
\end{theorem}

\begin{table}[!htbp]
\centering
\caption{Effective boundary conditions on $\Gamma \times (0, T)$ for the Dirichlet problem \eqref{fkpp: PDE}.}\label{fkpp: tb}
\begin{tabular}{l|lll}
As $\delta\to 0$   
    & $\frac{\sigma}{\delta}\to 0$ 
    & $\frac{\sigma}{\delta}\to\alpha\in(0,\infty)$ 
    &  $\frac{\sigma}{\delta}\to\infty$\\
    \hline
    \hline
    $\sigma\mu\to 0$  
    &  $\frac{\partial v}{\partial \textbf{n}}=0$	
    & $k\frac{\partial v}{\partial \textbf{n}}=-\alpha v$	
    &  $v=0$\\
    \hline
    $\sqrt{\sigma\mu}\to\gamma\in (0,\infty)$	
    &  $k\frac{\partial v}{\partial \textbf{n}}=\gamma \mathcal{J}^\infty [v]$ 
    &  $k\frac{\partial v}{\partial \textbf{n}}=\gamma \mathcal{J}^{\gamma/\alpha} [v]$	
    & $v=0$\\
    \hline
    $\sigma\mu\to \infty$
    & \makecell{ $\nabla_{\Gamma} v =0$,\\ $\int_{\Gamma}\frac{\partial v}{\partial\textbf{n}}=0 $}	
    & \makecell{$\nabla_{\Gamma} v =0$,\\ $\int_{\Gamma}(k\frac{\partial v}{\partial\textbf{n}}+\alpha v) =0 $}	
    &  $v=0 $  \\
    \hline
    \end{tabular}
\end{table}
The boundary condition $\nabla_\Gamma v =0$ on $\Gamma \times (0, T)$ implies that $v$ is a constant in $x$, where $\nabla_\Gamma$ represents the surface gradient on $\Gamma$;
$k\frac{\partial v}{\partial \textbf{n}}=\beta \Delta_\Gamma v $ can be interpreted as a second-order partial differential equation on $\Gamma$, where the operator $\Delta_\Gamma $ is the Laplacian-Beltrami defined on $\Gamma$, signifying that the thermal flux across $\Gamma$ in the outer normal direction results in heat accumulation that then diffuses with a rate $\beta$; $\mathcal{J}^{h}$ is the Dirichlet-to-Neumann mapping as defined in the following: for $ h \in(0, \infty)$, and smooth $g$ defined on $\Gamma$, $\mathcal{J}^{h}[g]$ is defined by
\begin{equation*}
    \mathcal{J}^h[g](s) := \Psi_R(s, 0),
\end{equation*}
where $\Psi$ is the bounded solution of 
 \begin{equation*}
    \left\{
             \begin{array}{ll}
                   \Psi_{RR}+\Delta_\Gamma      \Psi=0 , & \Gamma\times(0,H),\\
                  \Psi (s, 0) = g(s), &      \Psi(s, H)=0.
             \end{array}
  \right.
 \end{equation*}
The analytic formula for $\mathcal{J}^h[g]$ is given in the future, from which $\mathcal{J}^\infty[g]=\underset{h \to \infty}{\lim}\mathcal{J}^h[g]= - \left(-\Delta_\Gamma\right)^{1/2} g$ is the fractional Laplacian-Beltrami of order $1/2$. 

It is worth mentioning that the convergence of $u$ is valid only in the finite time interval $[0, T]$. Our interest lies in determining the maximal time interval on which the uniform convergence persists, referred to as the \textit{lifespan} of the EBC. The lifespan represents the duration during which the EBC remains effective. Notably, the lifespan of an EBC is not necessarily infinite. In the case of $f(u) \equiv 0$ in \eqref{fkpp: PDE}, it turns out that $u \to 0$ as $t \to \infty$, while $v \to \overline{u}_0$ when the EBC is Neumann, where $\overline{u}_0$ is the average of $u_0$ over $\Omega_1$. Thus, the lifespan of the Neumann EBC is not infinite if $\overline{u}_0 \neq 0$. In particular, Pond \cite{Pond2011} shows that if $\sigma = \mu =\delta^m$ with $m >2$, then the lifespan of the Neumann EBC is at the order of $o(\delta^{1-m})$ rather than $O(\delta^{1-m})$.

\smallskip

The remaining part of this paper is to prove the following theorem in which the lifespan of each EBC in Table \ref{fkpp: tb} is infinite.
\begin{theorem}\label{fkpp: thm2}
Suppose all conditions in Theorem \ref{fkpp: thm1} hold, then as $\delta \to 0$, the weak solution $u$ of \eqref{fkpp: PDE} satisfies
\begin{equation*}
\underset{0 \le t \le \infty}{\max} || u(\cdot, t) - v(\cdot, t) ||_{L^2(\Omega_1)} \longrightarrow  0,
\end{equation*}
where $v$ is the weak solution of \eqref{fkpp: EPDE} subject to any effective boundary condition in Table \ref{fkpp: tb}.
\end{theorem}
The proof of this theorem hinges on a comprehensive understanding of the non-negative steady states of \eqref{fkpp: PDE} and \eqref{fkpp: EPDE}. For general $f \neq 0$, the proof becomes complicated due to the impact of $f$ on altering the steady states of both $u-$problem and $v-$problem. However, in our specific case where $f=f(u)=u(1-u)$, this non-linearity facilitates the proximity between the two heat flows ($u-$flow and $v-$flow), even for large time.

There are lots of profound and intriguing findings in the literature regarding the study of EBCs. The original derivation of EBCs can be tracked back to the classic book by Carslaw and Jaeger \cite{HJ}. Subsequently, the elliptic problem involving reinforcement/coating was rigorously investigated by Brezis, Caffarelli, and Friedman \cite{BCF}, followed by Li and Zhang \cite{L2009, LZ} for further developments. Since then, numerous studies have emerged. See (\cite{CPW2012, GENG, G2023, LRWZ2009, LWW2014}) for the optimally aligned coating, \cite{HPW2013} for the model of wolf movement in which the diffusion tensor is degenerate along a road, and \cite{LL2022, LLW2022} for the error estimates between the solution of the original problem and the one of the effective problem. In the context of ecological problem in our case, see \cite{LSWW2021} for the derivation of bulk-surface model, \cite{LW2017} for the road-field model, and \cite{CHW2023} for the effect of an EBC. We also refer interested readers to a review paper by Wang \cite{W2016} on the study of EBCs for diffusion equations.

The organization of this paper is as follows. In Section \ref{sec: estimates}, some preliminaries are presented, including the definition of weak solutions of \eqref{fkpp: PDE} and \eqref{fkpp: EPDE},  and basic energy estimates from standard parabolic $L^2-$ theory. Section \ref{sec: EBCs} is devoted to deriving EBCs on $\Gamma \times (0, T)$ for the Fisher-KPP equation, in which an auxiliary function is proposed. In Section \ref{sec: lifespan}, based on the estimates for eigenvalue problems, we investigate the lifespan of EBCs by proving Theorem \ref{fkpp: thm2}.

Throughout this article, we always assume that $\Omega_1$ is fixed and bounded with $C^2$ smooth boundary $\Gamma$; the coating $\Omega_2$ is uniformly thick with $\partial\Omega$ approaching $\Gamma$ as $\delta \to 0$; $\sigma$ and  $\mu$ are positive functions of $\delta$.
\section{Preliminaries}\label{sec: estimates}
In the sequel, for notational convenience, let $C(T)$ represent a generic positive constant depending only on $T$ and $O(1)$ represent a quantity independent of $\delta$ but may vary from line to line.

\subsection{Eneregy estimates}
Let $W^{1,0}_2(Q_T)$ be the subspace of functions belonging to $L^2(Q_T)$ with first order weak derivatives in $x$ also being in $L^2(Q_T)$; $W^{1,1}_2(Q_T)$ is defined similarly with  the first order weak derivative in $t$ belonging to $L^2(Q_T)$; $W^{1,0}_{2,0}(Q_T)$ is the closure in $W^{1,0}_2(Q_T)$ of $C^\infty$ functions vanishing near $\overline {S}_T$, and $W^{1,1}_{2,0}(Q_T)$ is defined similarly. Furthermore, denote
$V^{1,0}_{2,0}(Q_T):=W^{1,0}_{2,0}(Q_T)\cap C  \left([0,T];L^2(\Omega)\right).$ 

Let us define one more Sobolev space on $Q^1_T= \Omega_1\times(0,  T): V^{1,0}_{2}(Q^1_T)=W^{1,0}_{2}(Q^1_T)\cap C  \left([0,T];L^2(\Omega_1)\right).$
We endow all these spaces with natural norms.

Rewrite $\int_0^T\int_{\Omega} u(x, t) dxdt$ as $\int_{Q_T} u(x, t) dxdt$ for simplicity. We begin with some energy estimates.
\begin{definition}\label{def}
A function $u$ is said to be a weak solution of the Dirichlet problem (\ref{fkpp: PDE}), if $u\in V^{1,0}_{2,0}(Q_T)$ and for any $\xi \in C^\infty (\overline{Q}_T) $ satisfying $\xi = 0$ at $t=T$ and also near $S_T$, it holds that
\begin{equation}
    \mathcal{A}[u,\xi] := -\int_{\Omega} u_0 \xi(x,0)dx + \int_{Q_T} \left( A(x)\nabla u \cdot \nabla\xi - u \xi_t-f(u)\xi \right)dxdt=0.  
 \end{equation}
\end{definition}

Note that $0$ and $M = \max\{ 1, ||u_0||_{L^\infty(\Omega)}\}$ are the corresponding lower and upper solution of \eqref{fkpp: PDE}. By the method of lower and upper solutions, for any small $\delta > 0$, (\ref{fkpp: PDE}) admits a unique weak solution $u \in V^{1,0}_{2,0}(Q_T)$, and $0 \leq u \leq M$ on $Q_T$. Moreover, $u$ satisfies the following ``transmission conditions" in the weak sense
\begin{equation}\label{trans}
  u_1=u_2, \quad k\nabla u_1 \cdot \textbf{n} = \sigma \nabla u_2 \cdot \textbf{n} \quad \text{ on }\Gamma,  
\end{equation}
where $u_1$ and $u_2$ are the restrictions of $u$ on $\Omega_1\times (0, T)$ and $\Omega_2\times (0, T)$, respectively. 

\begin{lemma}\label{fkpp: es}
Suppose $0 \leq u_0 \in L^\infty(\Omega)$. Then, any weak solution $u$ of  (\ref{fkpp: PDE}) satisfies the following inequalities.
\begin{equation*}
 \begin{split}
      \text{(i)}& \; 0 \leq u(x, t) \leq \max \{1 , ||u_0||_{L^\infty(\Omega)}\}. \\
     \text{(ii)}& \max_{t\in[0,T]} t\int_{\Omega}\nabla u \cdot A(x)\nabla udx +\int_{Q_T}t u_t^2dxdt
+\int_{Q_T}\nabla u \cdot A(x)\nabla u dxdt
     \leq C\left(T, |\Omega|, ||u_0||_{L^\infty(\Omega)}\right).\\
 \end{split}
\end{equation*}
\end{lemma}
\begin{proof}
$(i)$ can be proved by the lower and upper solutions method \cite{Pao2012} by choosing the lower solution 0 and the upper solution $u = \max \{1 , ||u_0||_{L^\infty(\Omega)}\}$. $(ii)$ follows from multiplying \eqref{fkpp: PDE} by $u$ and $t u_t$ respectively and performing the integration by parts in both $x$ and $t$ over $\Omega \times (0, T)$. By applying the same analysis to the Galerkin approximation of $u$, this formal argument can be rendered rigorous. Hence, we omit the details.
\end{proof}

\subsection{Weak solutions of effective problems}
We define weak solutions of (\ref{fkpp: EPDE}) together with the boundary conditions in Table \ref{fkpp: tb}.
\begin{definition}\label{def1}
Let the test function $\xi\in C^\infty (\overline{Q^1_T}) $ satisfy $\xi=0$ at $t=T$.

$(1)$ A function $v$ is said to be a weak solution of (\ref{fkpp: EPDE}) with the Dirichlet boundary condition $v=0$ if $v\in V^{1,0}_{2,0}(Q_T^1)$, and for any test function $\xi$, $v$ satisfies
\begin{equation}\label{wksol1}
\begin{split}
    \mathcal{L}[v,\xi] :=& -\int_{\Omega_1}u_0(x)\xi(x,0)dx+\int_0^T\int_{\Omega_1} \left(k\nabla v \cdot \nabla\xi-v\xi_t-f(v) \xi \right)dxdt
    =0.
\end{split}
\end{equation}

$(2)$ A function $v$ is said to be a weak solution of (\ref{fkpp: EPDE}) with the boundary conditions $\nabla_{\Gamma} v=0$ and
$\int_{\Gamma}(k\frac{\partial v}{\partial \textbf{n}}+\alpha v)=0$ for $\alpha \in [0, \infty)$ if for almost everywhere fixed $t\in (0, T)$, the trace of $v$ on $\Gamma$ is a constant, and if $\nabla_{\Gamma}\xi=0$ on $\Gamma$, it holds that $v\in V^{1,0}_2(Q_T^1)$ and $v$ satisfies
$$\mathcal{L}[v,\xi]=-\int_0^T\int_{\Gamma}\alpha v \xi dsdt.
$$

$(3)$ A function $v$ is said to be a weak solution of (\ref{fkpp: EPDE}) with the boundary condition $k\frac{\partial v}{\partial \textbf{n}}=\mathcal{B}[v]$, where $\mathcal{B}[v]=-\alpha v$, or $\gamma \mathcal{J}^h[v]$ for $h \in (0, \infty]$, if $v\in V^{1,0}_2(Q_T^1)$ and if for any test function $\xi$, $v$ satisfies
$$\mathcal{L}[v,\xi]=\int_0^T\int_{\Gamma} v \mathcal{B}[\xi] dsdt.
$$
\end{definition}
A weak solution of (\ref{fkpp: EPDE}) satisfies the initial value in the sense that $v(\cdot,t)\to u_0(\cdot)$ in $L^2(\Omega_1)$ as $t\to 0.$ 
\begin{theorem}\label{Uni}
Suppose that $\Gamma \in C^1$ and $0 \leq u_0\in L^\infty (\Omega_1)$. Then, \eqref{fkpp: EPDE} with any boundary condition in Tables \ref{fkpp: tb} has one and only one weak solution as defined in Definition \ref{def1}.
\end{theorem}
\begin{proof}
The theorem can be proved via the lower and upper solutions method by choosing the lower solution $0$ and the upper solution $v = \max \{1, || u_0||_{L^\infty(\Omega_1)} \}$.
\end{proof}

We end this section by introducing the following parametrization of the coating $\Omega_2$. Define $X$ by 
\begin{equation*}\label{curvilinear}
    \Gamma \times (0,  \delta) \longmapsto x = X(p, r) = p + r \textbf{n}(p) \in \mathbb{R}^3, 
\end{equation*}
where $p$ is the projection of $x$ on $\Gamma$; $\textbf{n}(p)$ is the unit normal vector of $\Gamma$ pointing out of $\Omega_1$ at $p$; $r$ is the distance from $x$ to $\Gamma$. As is well known (\cite[Lemma 14.16]{GT}), for a small $\delta>0$, $ X$ is a $C^1$ smooth diffeomorphism from $\Gamma \times (0, \delta)$ to $\Omega_2$; $r = r(x)$ is a $C^2$ smooth function of $x$ and is seen as the inverse of the mapping $x = X(p, r)$. 
 
By using local coordinates $s=(s_1, s_2)$ in a typical chart on $\Gamma$, we then have
\begin{equation}\label{cur}
    x = X(p(s), r) = X(s, r), \quad dx=(1+2Hr+\kappa r^2)dsdr  \quad \text{ in} \quad \overline{\Omega}_2,
\end{equation}
where $ds$ represents the surface element; $H(s) $and $\kappa(s) $ are the mean curvature and Gaussian curvature at  $p$ on $\Gamma$, respectively. Thus the Riemannian metric tensor at $x \in \overline{\Omega}_2$ induced from $ \mathbb{R}^3$ is defined as $G( s,r)$ with elements 
 $$g_{ij}(s, r)=g_{ji}(s, r) = < X_{i}, X_{j} >_{\mathbb{R}^3}, \quad i,j = 1,2,3,
 $$
where $X_i = X_{s_i}$ for $i =1,2$ and $X_3  = X_{r}$. Let $| G | :=  det G$ and $g^{ij}(s,r)$ be the element of the inverse matrix of $G$, denoted by $G^{-1}$. 

In the curvilinear coordinates $(s, r)$, the derivatives of $u$ are given as
\begin{equation}\label{derivative}
    \nabla u = u_r \textbf{n} +\nabla_s u, 
\end{equation}
where
\begin{equation}\label{derivative2}
  \begin{split}
     \nabla_s u  =\sum_{i,j=1,2} g^{ij}(s,r) u_{s_j}&X_{s_i}(s,r).
\end{split}
\end{equation}
In particular, if $r = 0$, then we have
$\nabla_{\Gamma}u=\sum_{i,j=1,2} g^{ij}(s,0)u_{s_j}p_{s_i}(s).$
Moreover, if $A(x)$ satisfies \eqref{OAC}, we have
\begin{equation}
     \nabla \cdot \left(A(x)\nabla u \right) = \frac{\sigma}{\sqrt{|G|}}\left(\sqrt{|G|} u_r\right)_r+\mu\Delta_{s}u,
\end{equation}
where
\begin{equation}
  \begin{split}
      \Delta_{s}u=\nabla_s \cdot \nabla_s u &=\frac{1}{\sqrt{|G|}}\sum_{ij=1,2}\left(\sqrt{|G|}g^{ij}(s, r)u_{s_i}\right)_{s_j}. 
  \end{split}
\end{equation}
Consequently, in $\overline{\Omega}_2$, we can rewrite $A(x)$ as
\begin{equation}\label{eq36}
    A(x) = \sigma \textbf{n}(p)\otimes\textbf{n}(p) + \mu \sum_{ij}g^{ij}(s,r)X_{s_i}(s,r)\otimes X_{s_j}(s,r).
\end{equation}

\section{EBCs for the Fisher-KPP equation}\label{sec: EBCs}
The aim of this section is to prove Theorem \ref{fkpp: thm1} in which EBCs are obtained on $\Gamma \times (0, T)$. 
\subsection{Auxiliary functions }
Enlightened by \cite{CPW2012}, an auxiliary function $\psi$ is constructed in $\Omega_2$. In the curvilinear coordinates $(s,r),$ for any $t \in [0, T]$ and smooth $g(s)$ defined on $\Gamma$, let $\psi(s, r, t)$ be the bounded solution of 
\begin{equation}\label{AF}
    \left\{
             \begin{array}{ll}
             \sigma      \psi_{rr}+\mu \Delta_\Gamma      \psi=0 , & \Gamma \times (0, \delta),\\
             \psi (s, 0, t) = g(s), &      \psi(s, \delta, t)=0.
             \end{array}
  \right.
\end{equation}
It follows from the maximum principle that $\psi$ is unique. Multiplying \eqref{AF} by $\psi$ and integrating by parts over $\Gamma \times (0, \delta)$, we arrive at
\begin{equation}\label{eq38}
      \int_0^\delta\int_{\Gamma} \left(\sigma     \psi_r^2+\mu |\nabla_\Gamma \psi|^2 \right)ds dr = -\int_{\Gamma} \sigma  \psi_r(s,0,t)g(s)ds.
\end{equation} 
Multiplying \eqref{AF} by $u$ and performing the integration by parts again, we get
\begin{equation}\label{eqno39}
    \begin{split}
     \int_0^\delta \int_{\Gamma} \left(\sigma     \psi_r u_r+\mu \nabla_\Gamma \psi \cdot \nabla_\Gamma u \right) dsdr &= -\int_{\Gamma} \sigma \psi_r(s,0,t) u(s,0, t)ds.
      \end{split}
\end{equation} 

We next assume $r = R\sqrt{\sigma/\mu}$ and $\Psi(s, R)= \psi(s, R\sqrt{\sigma/\mu}, t).$ Suppressing $t$ and plugging $r$ into \eqref{AF} lead to
\begin{equation}\label{rescale}
    \left\{
             \begin{array}{ll}
                   \Psi_{RR}+\Delta_\Gamma \Psi = 0, & \Gamma \times (0, h),\\
                  \Psi (s, 0) = g(s), & \Psi(s, h)=0,
             \end{array}
  \right.
\end{equation}
where $h=\delta\sqrt{\frac{\mu}{\sigma}}=\frac{\mu\delta}{\sqrt{\sigma\mu}}=\frac{\sqrt{\sigma\mu}}{\sigma/\delta}.$ Then the Dirichlet-to-Neumann operator is defined by
\begin{equation}\label{D2N}
\mathcal{J}^h[g](s) := \Psi_R(s,0),
\end{equation}
from which
\begin{equation}\label{eqno36}
\begin{split}
    \sigma \psi_r(s, 0, t) &= \sqrt{\sigma\mu}      \Psi_R(s,0) = \sqrt{\sigma\mu} \mathcal{J}^h[g](s). 
\end{split}
\end{equation}
By using separation of variables, $\mathcal{J}^h[g]$ can be given in eigenvalues and eigenfunctions of the operator $- \Delta_\Gamma$. In particular, we have
\begin{equation}\label{eq37}
	\begin{split}
             \Psi(s, R) & = \sum_{n=1}^\infty  \frac{-g_{n}e^{-\sqrt{\lambda_n} h}}{2 sinh(\sqrt{\lambda_n} h)} \left(e^{\sqrt{\lambda_n}R}-e^{\sqrt{\lambda_n}(2h-R)}\right)e_n(s),
    \end{split}
\end{equation}
where $ g_{n} = < e_n, g> = \int_{\Gamma} e_n g ds$; $\lambda_n$ and $e_n(s)$ are  the  eigenvalues and the corresponding eigenfunctions of $-\Delta_{\Gamma}$ defined on $\Gamma$. Thus, we obtain 
\begin{equation}\label{eq315}
\begin{split}
  \mathcal{J}^h[g](s) = &-\sum_{n=1}^\infty  \frac{\sqrt{\lambda_n} e_n(s)g_n}{tanh(\sqrt{\lambda_n}h)},
\end{split}
\end{equation}
which results from (\ref{D2N}) and  (\ref{eq37}).
Furthermore, if $h\to H \in (0, \infty ]$, we then have
\begin{equation*}
          \begin{split}
          \mathcal{J}^h[g](s)-\mathcal{J}^H[g](s)&=\sum_{n=1}^\infty \sqrt{\lambda_n} e_n(s)g_n \left(\frac{1}{tanh(\sqrt{\lambda_n}H)} -\frac{1}{tanh(\sqrt{\lambda_n}h)} \right) \\
           &=|H-h|\sum_{n=1}^\infty \lambda_n e_n(s)g_n\frac{-4}{(e^{\sqrt{\lambda_n}h^\prime}-e^{-\sqrt{\lambda_n}h^\prime})^2}\\
           &= O(|H-h|),
          \end{split} 
\end{equation*}
for some $h^\prime$ between $h$ and $H$, which means that $\mathcal{J}^h[g] \to \mathcal{J}^H[g]$ uniformly as $h \to H$, and $\mathcal{J}^\infty[g] = -(-\Delta_\Gamma)^{1/2}g.$

In what follows, we will estimate the size of the term $\Psi_R(s, 0)$ for a sufficiently small $\delta$. If $h \to 0$ as $\delta \to 0$, then it holds from (\ref{eq315}) that
\begin{equation*}
\begin{split}
    \left|\Psi_R(s, 0)+\frac{g(s)}{h}\right|& \leq  h \|g\|_{C^2(\Gamma)}.
\end{split}
\end{equation*}
Combining this with (\ref{eqno36}), we are led to
\begin{equation}\label{eq39}
 \begin{split}
     \sqrt{\sigma \mu} \Psi_R(s,0) &= \frac{\sigma}{\delta}\left(-g(s)+O(h^2)\right).
 \end{split}
\end{equation} 
On the other hand, if $h \to H \in (0, \infty]$ as $\delta\to 0$, then from the Taylor expansion for $\Psi(s, R)$, we obtain 
\begin{equation*}
         \Psi_R(s,0)=\frac{    \Psi(s, R)-     \Psi(s, 0)}{R}-\frac{R}{2}     \Psi_{RR}(s, \overline{R}), 
\end{equation*}
for some $\overline{R}\in [0,R]$. Then, it holds from the maximum principle that
\begin{equation*}
\begin{split}
    \| \Psi_R(s, 0)\|_{L^\infty(\Gamma))}&\leq\frac{2}{R}\|     \Psi\|_{L^\infty(\Omega_2)}+R\|      \Psi_{RR}\|_{L^\infty(\Omega_2)}\leq \frac{3\|g\|_{C^2(\Gamma)}}{R},
\end{split}
\end{equation*}
from which it turns out that 
\begin{equation}\label{eq310}
    \sqrt{\sigma\mu} \|\Psi_R\|_{L^\infty(\Gamma)}=\frac{O(1)\sqrt{\sigma\mu}}{R}. 
\end{equation}

\begin{remark}
For $h \in (0, \infty)$, $\mathcal{J}^h[g]$ is defined for smooth $g$. However, it is worth noting that it is also well-defined for $g \in H^{\frac{1}{2}}(\Gamma)$ where 
 $H^{\frac{1}{2}}(\Gamma)$ is defined by the completion of smooth functions under the $H^{\frac{1}{2}}(\Gamma)$ norm. Moreover, $\mathcal{J}^H: H^{\frac{1}{2}}(\Gamma) \to H^{-\frac{1}{2}}(\Gamma)$ is linear and symmetric with $H^{-\frac{1}{2}}(\Gamma)$ being the dual space of $H^{\frac{1}{2}}(\Gamma)$.
\end{remark}

\subsection{Main proof of Theorem \ref{fkpp: thm1}}
The main proof consists of two steps. Firstly, we establish the weak convergence of $u \to$ some $v$ in $W^{1,0}_{2}\left(\Omega_1\times(0,T)\right)$, and strong convergence of that in $C([0, T]; H^1(\Omega_1))$ after passing to a subsequence of $\delta > 0$. Subsequently, we use the auxiliary function $\psi$ to show that $v$ is a weak solution of \eqref{fkpp: EPDE} subject to the effective boundary conditions in Table \ref{fkpp: tb}.

\begin{proof}[Proof of Theorem \ref{fkpp: thm1}]
\textbf{Step 1.} We begin with the proof by considering the compactness of weak solutions $\{u\}_{\delta>0}$ of \eqref{fkpp: PDE}.

By Lemma \ref{fkpp: es}, $\{u\}_{\delta>0}$ is bounded in $W^{1,0}_{2}\left(\Omega_1\times(0,T)\right)$ and $C([t_0, T]; H^1(\Omega_1))$ for any small $t_0 \in (0, T]$. By Banach-Eberlein theorem, $ u \to $ some $v$ weakly in $W^{1,0}_{2}\left(\Omega_1\times(0,T)\right)$ and $C([t_0, T]; H^1_0(\Omega_1))$ after passing to a subsequence of $\delta \to 0.$ Since the embedding $H^1(\Omega_1) \hookrightarrow L^2(\Omega_1)$ is compact, for any fixed $t_0$, $\{u\}_{\delta>0}$ is pre-compact in $L^2(\Omega_1).$ Moreover, the functions $\{u\}_{\delta>0}$ : $t\in[t_0,T]\mapsto u(\cdot,t)\in L^2(\Omega_1)$ are equicontinuous because it holds from Lemma \ref{fkpp: es} that the term $\int_{Q_T}t u_t^2dxdt$ is bounded. Consequently, by the generalized Arzela-Ascoli theorem, after passing to a further subsequence of $\delta \to 0,$ $u \to v$ strongly in $C\left([t_0,T];L^2(\Omega_1)\right)$.

In the following, it suffices to prove that the strong convergence is in $C\left([0,T];L^2(\Omega_1)\right)$. We then take a sequence $u_0^n  \in C_0^\infty(\Omega_1)$ that can be constructed by multiplying $   u_0  $ by cut-off functions in the outer normal direction of $\Gamma$ such that $u_0^n \equiv 0$ in $\Omega_2$, $\|u_0-  u_0^n \|_{L^2(\Omega)}\leq \frac{1}{n}+\|u_0\|_{L^2(\Omega_2)}$,  and $\|\nabla  u_0^n \|_{L^2(\Omega)} \leq C(n) $. Then, we decompose $u=u_1+u_2$, where $u_1$ is the weak solution of \eqref{fkpp: PDE} with $f \equiv 0$ and the initial value replaced by $ u_0-u_0^n$, and $u_2$ is the weak solution of \eqref{fkpp: PDE} with the initial value replaced by $u_0^n$. 

Since the equation concerning $u_1$ is homogeneous, we have 
\begin{equation}\label{eq25}
\|u_1(\cdot, t)\|_{L^2(\Omega)} \leq \|u_0-  u_0^n \|_{L^2(\Omega)} \leq \frac{1}{n}+\|u_0\|_{L^2(\Omega_2)}.
\end{equation}
Multiplying the equation concerning $u_2$ by $(u_2)_t$ and performing integration by parts, we get 
\begin{equation*}
\begin{split}
   \int_0^t\int_\Omega(u_2)^2_t dxd\tau + \int_\Omega \nabla u_2(x,t) \cdot A(x) \nabla u_2(x,t)dx 
  \leq & \int_0^t\int_\Omega f^2(u) dxd\tau +\int_\Omega \nabla   u_0^n \cdot A(x) \nabla   u_0^n  dx\\
  \leq &\int_0^t\int_\Omega f^2(u) dxd\tau+k \int_{\Omega_1} |\nabla   u_0^n  |^2dx\\
  =:&F(f(u), n).
\end{split}
\end{equation*}
For any $t\in[0, t_0]$, it follows from this and (\ref{eq25}) that 
\begin{equation*}
\begin{split}
   \|u_2(\cdot,t)-  u_0^n  (\cdot)\|^2_{L^2(\Omega)} &=2\int_0^t\int_\Omega \left(u_2(x,t)-  u_0^n (x) \right) (u_2)_t dxd\tau\\
   & \leq 2\left(\int_0^t\int_\Omega \left(u_2(x,t)-  u_0^n (x)\right)^2 \right)^\frac{1}{2} \left(\int_0^t\int_\Omega (u_2)_t^2 \right)^\frac{1}{2}\\
&\leq 2\sqrt{t_0} \max_{t\in[0,t_0]} \|u_2(\cdot,t)-  u_0^n (\cdot)\|_{L^2(\Omega)}\left(F(f(u), n )\right)^\frac{1}{2},
\end{split}
\end{equation*}
resulting in
\begin{equation}\label{eq27}
    \max_{t\in[0,t_0]} \|u_2(\cdot,t)-  u_0^n (\cdot)\|\leq 2\sqrt{t_0}\left(F(f(u), n )\right)^\frac{1}{2}.
\end{equation}
Finally, it holds from (\ref{eq25}) and (\ref{eq27}) that
\begin{equation*}
\begin{split}
    \|u(\cdot,t)- u_0(\cdot)\|_{L^2(\Omega_1)} &\leq \|u_1(\cdot,t)\|_{L^2(\Omega_1)} +\|u_2(\cdot,t)- u_0(\cdot)\|_{L^2(\Omega_1)}+\|u_0-    u_0  (\cdot)\|_{L^2(\Omega_1)}\\
  \ &\leq \frac{2}{n}+2\|u_0\|_{L^2(\Omega_2)}+2\sqrt{t_0}\left(F(f(u), n )\right)^\frac{1}{2},
\end{split}
\end{equation*}
for any $n$ and sufficiently small $t_0$ and $\delta$. Thus, $\|u(\cdot,t)- u_0  (\cdot)\|_{L^2(\Omega_1)} $ can be arbitrary small for $t\in [0,t_0]$. Combining the fact that $u \to v$ strongly in $C\left([t_0,T];L^2(\Omega_1)\right)$, we conclude that $u \to v$ strongly in $C\left([0,T]; L^2(\Omega_1)\right)$  if we define $v(\cdot, 0) = u_0.$ 

\smallskip

\noindent
\textbf{Step 2.} In what follows, we aim to find all effective boundary conditions in Table \ref{fkpp: tb}.

By the compactness argument in Step 1, given any subsequence of $\delta$, we can ensure that $u \to $ some $v$ in all above spaces after passing to a further subsequence. In this step, we will show that $v$ is a weak solution of (\ref{fkpp: EPDE}) with effective boundary conditions listed in Table \ref{fkpp: tb}. Since $v$ is unique by Theorem \ref{Uni}, $u \to v$ without passing to any subsequence of $\delta >0$.

Let the test function $\xi \in C^{\infty}(\overline{\Omega}_1\times [0,T])$ satisfy $\xi=0$ at $t=T$, and extend $\xi$ to  $\overline{\Omega}\times [0,T]$ by defining
\begin{equation*}
   \overline{\xi}(x, t)=\left\{
\begin{aligned}
&\xi(x, t), &x\in \overline{\Omega}_1, \\
&\psi(p(x), r(x), t), &x\in\Omega_2,
\end{aligned}
\right.
\end{equation*}
where $\psi $ is introduced in \eqref{AF} with $g(s) = \xi(s, 0, t)$. It is easy to see $\overline{\xi}\in W^{1,1}_{2,0}(Q_T)$, and $\overline{\xi}$ is called the harmonic extension of $\xi$. According to Definition \ref{def}, it holds that
\begin{equation*}
\begin{split}
    \mathcal{A}[u,\overline{\xi}]&=-\int_{\Omega}u_0(x)\overline{\xi}(x,0)dx+\int_{0}^{T}\int_{\Omega} \left(\nabla \overline{\xi}\cdot A \nabla u-u\overline{\xi}_t-f(u) \overline{\xi} \right)dxdt=0,
\end{split}
\end{equation*}
which is equivalent to
\begin{equation}\label{eq314}
\begin{split}
     &\int_0^T\int_{\Omega_1} k\nabla\xi\cdot\nabla u dxdt-\int_\Omega u_0(x)\overline{\xi}(x,0)dx-\int_0^T\int_{\Omega} (u\overline{\xi_t} + f(u)\overline{\xi}) dxdt\\
  =&-\int_0^T\int_{\Omega_2}\nabla \psi \cdot A(x)\nabla u dxdt.
\end{split}
\end{equation}
By what we have proven in Step 1, $u \to v $ weakly in $W^{1, 0}_2\left(\Omega_1 \times(0,T) \right)$, and strongly in $C\left([0, T]; L^2(\Omega_1)\right)$ as $\delta \to 0$. Thus, it gives rise to
\begin{equation*}
\left\{
\begin{array}{ll}
     &  \int_{Q_T}u\xi_t dxdt \rightarrow \int_{Q^1_T}v\xi_t dxdt, \\
     & \int_{Q^1_T}\nabla u\cdot \nabla  \xi dxdt \to \int_{Q^1_T}\nabla v\cdot \nabla  \xi dxdt,\\
     &\int_{Q_T}f(u) \overline{\xi} dxdt \to \int_{Q^1_T}f(v) \xi dxdt.
\end{array}
\right. 
\end{equation*}
Then the left-hand side of (\ref{eq314}) is rewritten as
\begin{equation}\label{eqno53}
    \begin{split}
      \mathcal{L}[v, \xi] :=&\int_0^T\int_{\Omega_1} k\nabla\xi\cdot\nabla v dxdt-\int_{\Omega_1}u_0(x)\xi(x,0)dx-\int_0^T\int_{\Omega_1} (v\xi_t+f(v)\xi) dxdt. 
      \end{split} 
\end{equation}

The remainder of the following focuses on the right-hand side of (\ref{eq314}). In the curvilinear coordinates $(s,r)$, due to (\ref{cur}), (\ref{derivative}) and (\ref{eq36}), we have
\begin{equation}\label{RHS}
\begin{split}
     RHS:=&-\int_0^T\int_{\Omega_2}\nabla \psi \cdot A(x)\nabla u dxdt\\
     =&-\int_0^T\int_{\Gamma}\int_0^\delta \left(\sigma  \psi_r u_r+\mu\nabla_s \psi \nabla_s u \right) \left(1+2H(s)r+\kappa(s) r^2\right) drdsdt \\
    =&-\int_0^T\int_{\Gamma}\int_0^\delta \left(\sigma \psi_r u_r+\mu\nabla_\Gamma \psi \nabla_\Gamma u \right) -\int_0^T\int_{\Gamma}\int_0^\delta (\sigma     \psi_r u_r+\mu\nabla_\Gamma \psi \nabla_\Gamma u)(2Hr+\kappa r^2)  \\
    & -\int_0^T\int_{\Gamma}\int_0^\delta \mu (\nabla_s \psi \nabla_s u-\nabla_\Gamma \psi \nabla_\Gamma u)(1+2Hr+\kappa r^2)\\
    = & I+II+III.
\end{split} 
\end{equation}
Firstly, it follows from \eqref{eqno39} and \eqref{eqno36} that 
\begin{equation}\label{fkpp: I}
I = \sqrt{\sigma\mu}\int_0^T\int_{\Gamma}u(s, 0, t) \Psi_R(s, 0)dsdt.
\end{equation}
Secondly, in view of (\ref{eq38}),  (\ref{eqno36}) and Lemma \ref{fkpp: es}, we get
\begin{equation}\label{fkpp: II}
\begin{split}
    |II| = &\left|\int_0^T \int_{\Gamma}\int_0^\delta (\sigma     \psi_r u_r+\mu\nabla_\Gamma \psi \nabla_\Gamma u)(2Hr+\kappa r^2) drdsdt\right|\\
    = & O(\delta) \int_0^T \left( \int_{\Gamma}\int_0^\delta \sigma \psi_r^2 +\mu |\nabla_\Gamma \psi |^2 \right)^{1/2}\left( \int_{\Omega} \sigma u_r^2+\mu |\nabla_\Gamma u|^2 \right)^{1/2}dt\\
    = & O(\delta) \int_0^T \frac{1}{\sqrt{t}}\left(\int_\Gamma \sigma |\psi_r(s,0,t)| \right)^{1/2}dt\\
    = & O(\delta)  \sqrt{T}(\sigma\mu)^{1/4} ||\Psi_R(s,0)||^{1/2}_{L^\infty(\Gamma)},
\end{split}
\end{equation}
where we have used H\"oder inequality. Thirdly, according to \eqref{derivative2}, \eqref{eq38}, \eqref{eqno36} and Lemma \ref{fkpp: es}, it turns out that
\begin{equation}\label{fkpp: III}
\begin{split}
    |III|= & \left|\int_0^T\int_{\Gamma}\int_0^\delta \mu (\nabla_s     \psi \nabla_s u-\nabla_\Gamma \psi \nabla_\Gamma u)(1+2Hr+\kappa r^2) drdsdt\right|\\
    = & O(\delta)\left|\int_0^T\int_{\Gamma}\int_0^\delta \mu \sum_{ij} \psi_{s_i} u_{s_j} drdsdt\right|\\
    = & O(\delta) \int_0^T \left( \int_{\Gamma}\int_0^\delta \sigma \psi_r^2 +\mu |\nabla_\Gamma \psi |^2 \right)^{1/2}\left( \int_{\Omega} \sigma u_r^2+\mu |\nabla_\Gamma u|^2 \right)^{1/2}dt\\
    = & O(\delta)  \sqrt{T}(\sigma\mu)^{1/4} ||\Psi_R(s, 0)||^{1/2}_{L^\infty(\Gamma)}, 
\end{split}
\end{equation}
where the Taylor expansion for $g(s,r)$ and H\"oder inequality were used.

\smallskip

We next investigate the asymptotic behavior of RHS as $\delta \to 0$. To this end, we consider the following cases
$$
(1)  \frac{\sigma}{\delta}\to 0, \quad (2) \frac{\sigma}{\delta}\to \alpha\in (0,\infty),  \quad (3) \frac{\sigma}{\delta}\to \infty.
$$

\noindent
\textbf{Case $1$}. $\frac{\sigma}{\delta}\to 0$ as $\delta \to 0$. 

\smallskip
\noindent 
Subcase $(1i)$. $\sigma\mu \to 0$ as $\delta\to 0$. By the boundedness of $u$, \eqref{eq39} and \eqref{eq310}, \eqref{fkpp: I} can be rewritten as
$$|I| \le C(T) \max\{\frac{\sigma}{\delta}, (\sigma\mu)^{1/2} \}.
$$ 
Similarly, it holds from  \eqref{eq39} and \eqref{eq310} that 
$$ |II| + |III| = O(\delta) \sqrt{T} \max\{\sqrt{\frac{\sigma}{\delta}}, (\sigma\mu)^{1/4} \}.
$$
Combining these with \eqref{RHS}, we get
\begin{equation*}
    \begin{split}
      |RHS| \leq&  | I | + |II| + |III| \to 0 \; \text{ as } \; \delta \to 0,
    \end{split}
\end{equation*}
from which we have $\mathcal{L}[v,\xi]=0$, indicating that $v$ satisfies $\frac{\partial v}{\partial \textbf{n}}=0$ on $\Gamma \times (0, T)$.

\medskip
\noindent 
Subcase $(1ii)$. $\sqrt{\sigma\mu}\to \gamma \in (0, \infty)$ as $\delta\to 0$. In this case, $h\to \infty$. By the weak convergence of $u$, it follows from \eqref{eq315} and \eqref{fkpp: I} that
\begin{equation*}
    \begin{split}
      I  &   \longrightarrow \gamma \int_0^T \int_{\Gamma} v\mathcal{J}^\infty[g] dsdt \;\text{ as } \; \delta \to 0.
    \end{split}
\end{equation*}
Due to (\ref{eq310}), (\ref{fkpp: II}) and (\ref{fkpp: III}), it gives  
$$|II|+|III| \to 0 \; \text{ as } \; \delta\to 0,
$$
from which we are led to
$$ \mathcal{L}[v, \xi] = \gamma \int_0^T \int_{\Gamma}v\mathcal{J}^\infty[\xi].
$$ 
Thus, $v$ satisfies the effective boundary condition 
$k\frac{\partial v}{\partial \textbf{n}}=\gamma \mathcal{J}^\infty[v]$ on $\Gamma \times (0,T)$.

\medskip
\noindent 
Subcase $(1iii)$. $\sigma\mu \to \infty$ as $\delta\to 0$. In this case, $h \to \infty$. Divided both sides of \eqref{eq314} by  $\sqrt{\sigma\mu}$ and sending $\delta\to 0$, we get
\begin{equation*}
\int_0^T\int_{\Gamma}v \mathcal{J}^\infty[g] = 0.
\end{equation*}
Since the range of $\mathcal{J}^\infty[\cdot]$ contains $\{e_n\}_{n=1}^\infty$, it implies that $\nabla_\Gamma v =0$ on $\Gamma$ for almost everywhere $t\in (0, T)$.

Now, we further take the test function $\xi$ satisfying $\xi(s,0, t) = m(t)$ for some smooth function $m(t)$, and then define a new auxiliary function $\psi$ as 
$$\psi(s, r, t)= (1-\frac{r}{\delta}) m(t).
$$
Consequently, direct computation leads to
\begin{equation}\label{fkpp: lin. ext}
\begin{split}
    RHS =& -\int_0^T\int_{\Omega_2}\nabla \psi \cdot A\nabla u dxdt \\
    =&  \int_0^T \frac{\sigma m(t)}{\delta}\left(\int_0^\delta \int_{\Gamma} u_r (1+2Hr+\kappa r^2) dsdr\right)dt \\
    =& \int_0^T \frac{\sigma m(t)}{\delta}\left(\int_{\Gamma}u(s,0,t) ds \right) dt + \int_0^T \frac{\sigma m(t)}{\delta}\int_0^\delta \int_{\Gamma} u (2H+2\kappa r)dsdr \\
   = &  \frac{\sigma }{\delta}\int_0^T m(t)dt \left( O(1) + O(\delta)\|u \|_{L^\infty(Q_T)}\right) \to 0 \; \text{ as } \; \delta \to 0,
\end{split}
\end{equation} 
from which we have $\mathcal{L}[v,\xi]=0$ as $\delta \to 0$. Thus, $v$ satisfies $\int_{\Gamma}\frac{\partial v}{\partial \textbf{n}} ds =0$ on $\Gamma \times (0,T)$ due to the arbitrariness of the test function $\xi$.

\medskip
\noindent
 \textbf{Case $2$}. $\frac{\sigma}{\delta}\to \alpha \in (0, \infty)$ as $\delta\to 0$. 

\noindent 
Subcase $(2i)$. $\sigma\mu \to 0$ as $\delta\to 0$. In this case,  $h \to 0$. It follows from \eqref{eq39} and \eqref{fkpp: I} that
\begin{equation*}
    I \to -\alpha\int_0^T\int_{\Gamma}v\xi \; \text{ as } \delta\to 0.
\end{equation*}
Subsequently, in view of \eqref{eq39}, \eqref{fkpp: II} and \eqref{fkpp: III}, we have $$
| II + III| \to 0 \; \text{ as } \; \delta \to 0,
$$
from which we are led to $\mathcal{L}[v,\xi]=-\alpha\int_0^T\int_{\Gamma}v\xi,$ meaning that $v $ satisfies $k\frac{\partial v}{\partial \textbf{n}}=-\alpha v$ on $\Gamma\times (0,T)$.

\smallskip
\noindent Subcase $(2ii)$. $\sqrt{\sigma\mu}\to \gamma \in (0, \infty)$ as $\delta\to 0$.
In this case, $h \to H = \gamma/\alpha \in (0, \infty)$.  By virtue of (\ref{eq39}) and (\ref{RHS})- (\ref{fkpp: III}), it holds that 
\begin{equation*}
    I \longrightarrow \gamma\int_0^T\int_{\Gamma}v \mathcal{J}^{\gamma/\alpha}[\xi] \quad \text{ and } \quad |II+III| \longrightarrow 0 \text{  as  } \delta \to 0,
\end{equation*}
from which we get $\mathcal{L}[v,\xi]=\gamma\int_0^T\int_{\Gamma}v\mathcal{J}_D^{\gamma/\alpha}[\xi],$ implying that $v$ satisfies  $ k\frac{\partial v}{\partial \textbf{n}}=\gamma \mathcal{J}^{\gamma/\alpha}[v]$ on $\Gamma\times (0,T)$.

\medskip
\noindent Subcase $(2iii)$. $\sigma\mu \to \infty$ as $\delta\to 0$. In this case, $h \to \infty$. Divided both sides of (\ref{eq314}) by  $\sqrt{\sigma\mu}$ and sending $\delta\to 0$, we obtain $$\int_0^T\int_{\Gamma}v\mathcal{J}^\infty[\xi]=0,
$$ 
from which we get $\nabla_\Gamma v =0$ on $\Gamma$ by using the same argument as in Subcase $(1iii)$. Then take the same test function $\xi$ and the same auxiliary function $\psi$ as in Subcase $(1iii)$, and thus it follows from \eqref{fkpp: lin. ext} that
$$ \mathcal{L}[v,\xi]=-\alpha\int_0^T\int_{\Gamma}v\xi,
$$
which means $v$ satisfies  $\int_{\Gamma}\left(k\frac{\partial v}{\partial \textbf{n}}+\alpha v \right)=0$
on $\Gamma\times(0,T)$.

\medskip
\noindent
\textbf{Case $3$}.  $\frac{\sigma}{\delta}\to \infty$ as $\delta\to 0$.

\noindent 
Subcase $(3i)$. $\sqrt{\sigma\mu}\to \gamma \in [0, \infty)$ as $\delta\to 0$. In this case,  $h \to 0$. Divided both sides of (\ref{eq314}) by $\sigma / \delta$ and sending $\delta \to 0$, in view of (\ref{eq38}) and (\ref{RHS})- (\ref{fkpp: III}), we are led to 
\begin{equation*}
 \frac{\delta}{\sigma} I \longrightarrow -\int_0^T\int_{\Gamma}v\xi = 0, 
\end{equation*}
from which $v=0$ on $\Gamma\times(0,T)$ since $\xi$ is arbitrary. 

\smallskip
\noindent 
Subcase $(3ii)$. $\sigma\mu \to \infty$ as $\delta\to 0$. In this case, after passing to a subsequence, we have $h\to H \in[0, \infty]$. If $H=0$, then divided both sides of (\ref{eq314}) by $\sigma / \delta$ and sending $\delta\to 0$, we have  $v=0$ on $\Gamma\times (0, T)$. 

On the other hand, if $H\in (0, \infty]$, then divided both sides of (\ref{eq314}) by $\sqrt{\sigma\mu}$ and sending $\delta \to 0$, we have
\begin{equation*}
 \frac{I}{\sqrt{\sigma\mu}}\longrightarrow \int_0^T\int_{\Gamma}v \mathcal{J}^H[\xi]=0.
\end{equation*}
By using the same argument as in Subcase $(1iii)$, for almost everywhere $t\in (0, T)$, we have $\nabla_\Gamma v = 0$, and
$$
\int_0^T\int_{\Gamma}vm(t)=0,
$$ 
from which we have $v=0$ on $\Gamma\times(0,T)$ for $m(t)$ is an arbitrary function in $t$.

\smallskip
Therefore, we complete the whole proof of Theorem \ref{fkpp: thm1}.
\end{proof}

\section{Lifespan of EBCs}\label{sec: lifespan}

In this section, we delve into the lifespan of each EBC in Table \ref{fkpp: tb} to determinate how long these EBCs remain effective. Moreover, we prove that the lifespan of each EBC is infinite, which means that the convergence of $u(\cdot, t) \to v(\cdot, t)$ in $L^2(\Omega_1)$ is global in $t$.

\subsection{Eigenvalue problems}

Consider the eigenvalue problem of
\begin{equation} \label{EP}
    \left\{
             \begin{array}{llr}
             -\nabla \cdot (A(x)\nabla u) = \lambda u, &\mbox{$x \in \Omega,$}\\
             u=0,     &\mbox{$ x \in \partial \Omega,$}& \\
             \end{array}
  \right.
\end{equation}
where $A(x)$ is given in \eqref{A} and \eqref{OAC}. Let $( \lambda_1^\delta, e_1^\delta )$ be an eigenpair of \eqref{EP} such that $e_1^\delta > 0$ in $\Omega$ and $||e_1^\delta ||_{L^2(\Omega)} = 1$. It is well-known that 
\begin{equation}\label{fkpp: eq42}
  \lambda_1^\delta =   \underset{0 \not\equiv u \in H^1_0(\Omega)}{\inf} \frac{\int_\Omega A \nabla u \cdot \nabla u dx}{\int_\Omega u^2 dx}.
\end{equation} 

\begin{lemma}\label{fkpp: le41}
Suppose that $\frac{\sigma}{\delta} \to 0$ as $\delta \to 0$. Then, for any small $\delta > 0$, it holds,
$$ C_1 \frac{\sigma}{\delta} \leq \lambda_1^\delta \leq C_2 \frac{\sigma}{\delta},
$$
and 
$$\int_{\Omega_1}\left( e_1^\delta - \frac{1}{\sqrt{|\Omega_1|}}\right)^2 dx \leq C \left( \frac{\sigma}{\delta} + \delta^2 \right),
$$
where $C_1, C_2$ and $C$ are positive constants independent of $\delta$.
\end{lemma}
\begin{proof}
The can be found in \cite[Lemma 2.1]{LLW2022}, and hence we omit the details.
\end{proof}

\begin{theorem}\label{fkpp: thm3}
Suppose $A(x)$ is given in \eqref{A} and \eqref{OAC}. Assume that $\sigma$ and $\mu$ satisfy 
\begin{equation*}
    \lim_{\delta\to 0}\sigma\mu=\gamma\in[0,\infty], \quad
\lim_{\delta\to 0}\frac{\sigma}{\delta}=\alpha\in[0,\infty].
\end{equation*}
Then, as $\delta \to 0$,
$$
\lambda_1^\delta \to \lambda_1 \quad \text{ and } \quad e_1^\delta \to e_1 \quad \text{ in } L^2(\Omega_1),
$$
where $e_1 > 0$ with $||e_1||_{L^2(\Omega_1)} = 1$, and $(\lambda_1, e_1)$ is the principal eigenpair of
\begin{equation}\label{fkpp: ep2}
- k\Delta v = \lambda v, \quad  x \in \Omega_1
\end{equation}
subject to the effective boundary conditions listed in Table \ref{fkpp: tb}.
\end{theorem}\label{fkpp: ep3}

\begin{proof}[Proof of Theorem \ref{fkpp: thm3}]
Denote by $(\lambda_1^D, e_1^D)$ the principal eigenpair of \eqref{fkpp: ep2} subject to the Dirichlet boundary condition. Taking a special test function $\phi \in H^1_0(\Omega)$ in \eqref{fkpp: eq42} as
\begin{equation} 
    \phi = \left\{
             \begin{array}{llr}
             e_1^D, &\mbox{$x \in \Omega_1,$}\\
             0,     &\mbox{$ x \in\Omega_2,$}& 
             \end{array}
  \right.
\end{equation}
we are led to  
$$
\lambda_1^\delta < \frac{k\int_{\Omega_1}  |\nabla e_1^D|^2  dx}{\int_{\Omega_1} (e_1^D)^2 dx} = \lambda_1^D.
$$
Moreover, it holds from \eqref{fkpp: eq42} that $\lambda_1^\delta > 0$. Then after passing to a subsequence of $\delta > 0$, we have $$\lambda_1^\delta \to \; \text{ some } \; \lambda_1^0.
$$ 
Subsequently, it is easy to see that
\begin{equation}\label{fkpp: eq45}
   \lambda_1^\delta =   k\int_{\Omega_1} |\nabla e_1^\delta|^2 dx + \int_{\Omega_2} \sigma  (e_1^\delta)_r^2 + \mu |\nabla_s e_1^\delta|^2  dx \le \lambda_1^D,
\end{equation} 
from which $\{e_1^\delta\}_{\delta > 0}$ is bounded in $H^1(\Omega_1)$.
Thus, after passing to a further subsequence of $\delta \to 0$, we have 
$$e_1^\delta \to \; \text{some} \; e_1^0
$$ 
weakly in $H^1(\Omega_1)$ and strongly in $L^2(\Omega_1)$. 

Consequently, similar to the parabolic case in the proof of Theorem \ref{fkpp: thm1}, we can show that $e_1^0$ is a weak solution of \eqref{fkpp: ep2} subject to the effective boundary conditions in Table \ref{fkpp: tb} with the corresponding $\lambda_1^0$. We next prove that $e_1^0 \not\equiv 0$, after which it holds that $e^0_1$ is the eigenfunction of $\lambda_1^0$. 

In view of \eqref{cur}, we get
\begin{equation}\label{fkpp: eq46}
 \begin{split}
    \int_{\Omega_2} (e_1^\delta)^2 dx = &\int_{\Gamma}\int_0^\delta (e_1^\delta)^2 (1+H(s)r+\kappa(s)r^2)drds\\
    = & \int_{\Gamma}\int_0^\delta \left(\int_r^\delta (e_1^\delta)_\tau  d\tau \right)^2 (1+H(s)r+\kappa(s)r^2)drds\\
    \le &C \delta \int_{\Gamma} \int_0^\delta \left(\int_0^\delta (e_1^\delta)^2_\tau  d\tau \right) (1+H(s)r+\kappa(s)r^2)drds\\
    \le & C\delta^2 \int_{\Gamma} \int_0^\delta (e_1^\delta)^2_r drds,
\end{split}
\end{equation}
from which, together with \eqref{fkpp: eq45}, we have 
$$\int_{\Omega_2} (e_1^\delta)^2 dx \le C\frac{\delta^2}{\sigma} \to 0 \; \text{ as } \; \delta \to 0,
$$
in the case of $\frac{\sigma}{\delta} \to \alpha \in (0, \infty]$. Combining \eqref{fkpp: eq45} and \eqref{fkpp: eq46} with Lemma \ref{fkpp: le41}, as $\delta \to 0$, we have $$\int_{\Omega_2} (e_1^\delta)^2 dx \le C\delta \to 0
$$ in the case of $\frac{\sigma}{\delta} \to 0$. Thus, we obtain $e_1^0 \not\equiv 0$ and $||e_1^0||_{L^2(\Omega_1)} = 1$, resulting from
$$1 = \int_\Omega (e_1^\delta)^2 dx = \int_{\Omega_1} (e_1^\delta)^2 dx + \int_{\Omega_2} (e_1^\delta)^2 dx.
$$
Consequently, by the uniqueness of the principal eigenpair, it follows that $(\lambda_1, e_1) = (\lambda_1^0, e_1^0)$ and all above convergences hold without passing to a subsequence.
\end{proof}
We next consider the eigenvalue problems of the operator $-k\Delta$ subject to some nonlocal boundary conditions in Table \ref{fkpp: tb}.
\begin{lemma} \label{fkpp: le42}
(i) Consider the eigenvalue problem of
\begin{equation} \label{EP2}
    \left\{
             \begin{array}{llr}
             -k \Delta v = \mu v, &\mbox{$x \in \Omega_1,$}\\
             k\frac{\partial v}{\partial \textbf{n}} =\gamma \mathcal{J}^{\infty}[v],     &\mbox{$ x \in \partial \Omega_1,$}& \\
             \end{array}
  \right.
\end{equation}
where $\alpha = 0$ and $\gamma \in (0, \infty)$. Then, the principle eigenvalue of \eqref{EP2} is $0$, and the second eigenvalue is positive.

(ii) Consider the eigenvalue problem of
\begin{equation} \label{EP4}
    \left\{
             \begin{array}{llr}
             -k \Delta v = \mu v, &\mbox{$x \in \Omega_1,$}\\
            \nabla_\Gamma v =0, \;\int_{\partial\Omega_1} \frac{\partial v}{\partial \textbf{n}} ds = 0,     &\mbox{$ x \in \partial \Omega_1,$}& \\
             \end{array}
  \right.
\end{equation}
where $\alpha = 0$ and $\gamma = \infty$. Then, the principle eigenvalue of \eqref{EP4} is $0$, and the second eigenvalue is positive.

(iii) Consider the eigenvalue problem of
\begin{equation} \label{EP1}
    \left\{
             \begin{array}{llr}
             -k \Delta v = \mu v, &\mbox{$x \in \Omega_1,$}\\
             k\frac{\partial v}{\partial \textbf{n}} =\gamma \mathcal{J}^{\gamma/\alpha}[v],     &\mbox{$ x \in \partial \Omega_1,$}& \\
             \end{array}
  \right.
\end{equation}
where $\alpha \in (0, \infty)$ and $\gamma \in (0, \infty)$. Then, the principle eigenvalue of \eqref{EP1} is positive.

(iv) Consider the eigenvalue problem of
\begin{equation} \label{EP3}
    \left\{
             \begin{array}{llr}
             -k \Delta v = \mu v, &\mbox{$x \in \Omega_1,$}\\
             \nabla_\Gamma v =0, \; \int_{\partial \Omega_1}\left( k \frac{\partial v}{\partial \textbf{n}} +\alpha v\right) ds = 0,     &\mbox{$ x \in \partial \Omega_1,$}& \\
             \end{array}
  \right.
\end{equation}
where $\alpha \in (0, \infty)$ and $\gamma = \infty$ . Then, the principle eigenvalue of \eqref{EP3} is positive.
\end{lemma}
\begin{proof}
The proof of this Lemma can be found in \cite[Lemma 2.2-2.3]{LLW2022}, and hence we omit the details.
\end{proof}
 
We now turn to the positive steady state of \eqref{fkpp: PDE} as
\begin{equation} \label{fkpp: ps1}
    \left\{
             \begin{array}{llr}
             -\nabla \cdot (A(x)\nabla U) = U(1 - U), &\mbox{$x \in \Omega,$}\\
             U=0,     &\mbox{$ x \in \partial \Omega,$}& \\
             \end{array}
  \right.
\end{equation}
where $U$ is the unique positive solution. It is well-known that $(i)$ the positive steady state $U$ exists if and only if the principal eigenvalue $\lambda_1^\delta < 1$; $(ii)$ if $U$ exists, then it is unique; $(iii)$ as $t \to \infty$, $u \to U$ if $U$ exists ($u \to 0$ if $U$ does not exist).

\begin{theorem}\label{fkpp: prop}
Suppose that $A(x)$ is given in \eqref{A} and \eqref{OAC}. Assume that $\sigma$ and $\mu$ satisfy 
\begin{equation*}
    \lim_{\delta\to 0}\sigma\mu=\gamma\in[0,\infty], \quad
\lim_{\delta\to 0}\frac{\sigma}{\delta}=\alpha\in[0,\infty].
\end{equation*}
Then, as $\delta \to 0$, $U \to V$ in $L^2(\Omega_1)$, where $V$ is the unique positive solution of 
\begin{equation} \label{fkpp: ps2}
- k\Delta V = V(1 - V), \;  x \in \Omega_1
\end{equation}
subject to the effective boundary conditions listed in Table \ref{fkpp: tb} (with $v$ replaced by $V$).
\end{theorem}
\begin{remark}\label{fkpp: rk}
As the results about $U$, it holds that the positive steady state $V$ exists if and only if the principal eigenvalue $\lambda_1 < 1$; if $V$ exists, then it is unique; as $t \to \infty$, $v \to V$ if $V$ exists ($v \to 0$ if $V$ does not exist). Moreover, if $\alpha = 0$, then $V \equiv 1$.
\end{remark}

\begin{proof}[Proof of Theorem \ref{fkpp: prop}]
By performing the process similar to the parabolic case in the proof of Theorem \ref{fkpp: thm1}, we can show that after passing to a subsequence of $\delta \to 0$,
$$
U \to  \; \text{ some } \; U_0 \; \text{ in } \; L^2(\Omega_1),
$$
where $U_0$ is a weak solution of \eqref{fkpp: ps2} ($U_0$ may be 0 in $\Omega_1$). It is easy to see that the weak solution of \eqref{fkpp: ps2} is unique. Thus, all above convergences hold without passing to a subsequence. We next will show that $U_0 \not\equiv 0$ to ensure that $U_0$ is positive, namely, $U_0>0$ in $\Omega_1$. Thus, we have $U_0 = V$.

In what follows, it suffices to prove that $U_0 \not\equiv 0$ due to the fact $U \ge 0$. In the case of $ \frac{\sigma}{\delta} \to 0$ as $\delta \to 0$, we can construct a lower solution by defining $$\underline{U}(x) = \varepsilon e_1^\delta (x)
$$ 
where $\varepsilon > 0$ is sufficiently small. For small $\delta$, it holds from Lemma \ref{fkpp: le41} that 
$$  -\nabla \cdot (A(x)\nabla \underline{U}) - \underline{U}(1 - \underline{U}) = \underline{U}(\lambda_1^\delta - 1 + \varepsilon e_1^\delta) \le 0.
$$
Since $U \equiv 1$ is a upper solution, it follows from the lower and upper solutions method that 
$$U \ge \underline{U} = \varepsilon e_1^\delta,
$$
from which we have $U_0 \ge \underline{U} > 0$ in $\Omega_1$.

\smallskip
On the other hand, in the case of $ \frac{\sigma}{\delta} \to \alpha \in (0, \infty]$ as $\delta \to 0$, we will prove $U_0 > 0$ by contradiction. If not, then $U_0 \equiv 0$. Thus, let $W = \frac{U}{||U||_{L^2(\Omega)}} > 0$, and by \eqref{fkpp: ps1}, $W$ satisfies the following elliptic problem
\begin{equation*} 
    \left\{
             \begin{array}{llr}
             -\nabla \cdot (A(x)\nabla W) = W(1 - U) , &\mbox{$x \in \Omega,$}\\
             W=0,     &\mbox{$ x \in \partial \Omega.$}& \\
             \end{array}
  \right.
\end{equation*}
Multiplying both sides by $W$ and performing integration by parts, we have
\begin{equation} 
  k \int_{\Omega_1} |\nabla W|^2 dx + \int_{\Omega_2} A(x)\nabla W \cdot \nabla W dx + \int_{\Omega}W^2 (U - 1) dx = 0,
\end{equation}
from which $\{W\}_{\delta > 0 }$ is bounded in $H^1_0(\Omega_1)$. As a result, after passing to a subsequence of $\delta$, $W \to $ some $W_0$ weakly in $H^1_0(\Omega_1)$ and strongly in $L^2(\Omega_1)$. Moreover, similar to \eqref{fkpp: eq46}, we get
\begin{equation}
 \begin{split}
    \int_{\Omega_2} W^2 dx = &\int_{\Gamma}\int_0^\delta W^2 (1+H(s)r+\kappa(s)r^2)drds\\
    \le & C(\delta^2) \int_{\Gamma} \int_0^\delta W^2_r drds\\
    = & O(\frac{\delta^2}{\sigma}) \to 0 \; \text{ as } \; \delta \to 0.
\end{split}
\end{equation}
Consequently, we have $|| W_0 ||_{L^2(\Omega_1)} = 1$ because of the fact $|| W ||_{L^2(\Omega)} = 1$. Following the process as in the parabolic case in the proof of Theorem \ref{fkpp: thm1}, we can show that $W_0$ is a weak solution of
\begin{equation*} 
-k \Delta W_0 = W_0, \quad  x \in \Omega_1,
\end{equation*}
subject to the corresponding effective boundary conditions in Table \ref{fkpp: tb} (with $v$ replaced by $W_0$). This means that $\lambda_1 = 1$, which contradicts the necessary condition for the existence of $V$.

\end{proof}

\subsection{Main proof of Theorem \ref{fkpp: thm2} }
In the remaining part, our primary objective is to establish the proof of Theorem \ref{fkpp: thm2}, which unveils the infinite lifespan of each EBC as in Table \ref{fkpp: tb}.

\begin{proof}[Proof of Theorem \ref{fkpp: thm2}]
If the initial value $u_0  \equiv 0$, then it is easy to see that $u = v = 0$ in $\Omega_1 \times (0, \infty)$, which implies this theorem trivially. Thus, we always have the assumption of $u_0 \not\equiv 0$.

\smallskip
\noindent
\textbf{Case 1.} $\frac{\sigma}{\delta} \to 0$ as $\delta \to 0$, i.e., $\alpha = 0$. 

According to Theorem \ref{fkpp: thm3}, as $\delta \to 0$, it holds that $\lambda_1^\delta \to \lambda_1$, where $\lambda_1$ is the principal eigenvalue of \eqref{fkpp: ep2} with the corresponding effective boundary conditions (second column in Table \ref{fkpp: tb}). By Lemma \ref{fkpp: le41}, it turns out that $\lambda_1 = 0$. Thus, for a sufficiently small $\delta > 0$, it is easy to see that $\lambda_1^\delta < 1$, implying that \eqref{fkpp: ps2} has a unique positive steady state solution $U$. Moreover, it follows from Theorem \ref{fkpp: prop} and Remark \ref{fkpp: rk} that as $\delta \to 0$, $U \to V$ in $L^2(\Omega_1)$ with $V \equiv 1$.

Subsequently, we consider the eigenvalue problem of
\begin{equation} 
    \left\{
             \begin{array}{llr}
             -\nabla \cdot (A(x)\nabla u) - u(1 - U) = \eta u, &\mbox{$x \in \Omega,$}\\
             u=0,     &\mbox{$ x \in \partial \Omega,$}& \\
             \end{array}
  \right.
\end{equation}
from which $(\eta_1, u) = (0, U)$ is the principal eigenpair with $U > 0$ in $\Omega$. Let $\widetilde{u} = u - U$ and then $\widetilde{u}$ satisfies
\begin{equation} \label{fkpp: eq416}
    \left\{
             \begin{array}{llr}
             \widetilde{u}_t-\nabla \cdot (A(x)\nabla \widetilde{u}) = \widetilde{u}(1 -u- U), &\mbox{$(x, t) \in \Omega \times (0, \infty),$}\\
            \widetilde{u} = u_0 - U,     &\mbox{$ (x, t) \in \partial \Omega \times (0, \infty).$}& \\
             \end{array}
  \right.
\end{equation}
Multiplying \eqref{fkpp: eq416} by $\widetilde{u}$ and integrating by parts in $x$, we have
\begin{equation} \label{fkpp: eq417}
   \begin{split}
        \frac{1}{2} \frac{d}{dt}\int_{\Omega}\widetilde{u}^2dx =&-\int_{\Omega} \nabla \widetilde{u} \cdot A \nabla \widetilde{u} + \widetilde{u}^2(1- U)dx - \int_{\Omega}\widetilde{u}^2 u  dx\\
        \le & - \int_{\Omega} \widetilde{u}^2 u dx\\
        \le& 0,
   \end{split}
\end{equation}
which results from Lemma \ref{fkpp: es} and the variational characterization of $\eta_1$ as 
\begin{equation}
    0 = \eta_1 =  \underset{0 \not\equiv u \in H^1_0(\Omega)}{\inf} \frac{\int_\Omega A \nabla u \cdot \nabla u + u^2(U-1)dx}{\int_\Omega u^2 dx}.
\end{equation}
\end{proof}
Consequently, given a fixed $T > 0$, and for any $t \ge T$, \eqref{fkpp: eq417} gives
\begin{equation} \label{fkpp: eq419}
   \begin{split}
        ||u(\cdot, t) - U||_{L^2(\Omega_1)}
        \le &||u(\cdot, t) - U||_{L^2(\Omega)}\\
        \le& ||u(\cdot, T) - U||_{L^2(\Omega)} \\
        \le &||u(\cdot, T) - v(\cdot, T)||_{L^2(\Omega_1)} +||v(\cdot, T) - V||_{L^2(\Omega_1)} \\
        &+ ||U - V||_{L^2(\Omega_1)}
        +||u(\cdot, T) - U||_{L^2(\Omega_2)}.
   \end{split}
\end{equation}
Since $\lambda_1 = 0$ and $u_0 \ge 0, \not\equiv 0$, as $ t \to \infty$, it holds that
$$ v(x,t) \to V \equiv 1 \;  \text{ uniformly in } \; \overline{\Omega}_1,
$$
which means that for any $\varepsilon > 0$, there exists some large $T_\varepsilon > 0$ such that 
$$ |v(x,t) - V(x)| < \varepsilon, \quad  \forall x\in \overline{\Omega}_1,
$$
provided $t \ge T_\varepsilon$. Thus, we have 
$$
||v(\cdot, t) - V||_{L^2(\Omega_1)} \le |\Omega_1|^{1/2} \varepsilon
$$ 
for $t \ge T_{\varepsilon}$. Furthermore, for a sufficiently small $\delta > 0$, thanks to Theorem \ref{fkpp: thm1} and Lemma \ref{fkpp: es}, we are led to
\begin{equation} 
   \begin{split}
      ||u(\cdot, T_\varepsilon) - v(\cdot, T_\varepsilon)||_{L^2(\Omega_1)} \le  \varepsilon,
   \end{split}
\end{equation}
and
\begin{equation} \label{fkpp: eq421}
   \begin{split}
      ||u(\cdot, T_\varepsilon) - U||_{L^2(\Omega_2)} \le   \varepsilon.
   \end{split}
\end{equation}
As a result, for any $t\ge T_\varepsilon$, due to Theorem \ref{fkpp: prop} and above estimates, \eqref{fkpp: eq419} gives
\begin{equation} 
   \begin{split}
        ||u(\cdot, t) - U||_{L^2(\Omega_1)} 
        \le &||u(\cdot, T_\varepsilon) - v(\cdot, T_\varepsilon)||_{L^2(\Omega_1)} +||v(\cdot, T_\varepsilon) - V||_{L^2(\Omega_1)} \\
        &+ ||U - V||_{L^2(\Omega_1)}
        +||u(\cdot, T) - U||_{L^2(\Omega_2)}\\
        \le & C\varepsilon.
   \end{split}
\end{equation}
Therefore, in view of \eqref{fkpp: eq419}-\eqref{fkpp: eq421}, we have
\begin{equation} 
   \begin{split}
        &\underset{t\in[T_\varepsilon, \infty]}{\max}||u(\cdot, t) - v(\cdot, t) ||_{L^2(\Omega_1)}\\ \le&\underset{t\in[T_\varepsilon, \infty]}{\max}||u(\cdot, t) - U ||_{L^2(\Omega_1)} +  \underset{t\in[T_\varepsilon, \infty]}{\max}||v(\cdot, t) - V ||_{L^2(\Omega_1)} + ||U - V ||_{L^2(\Omega_1)} \\
        \le&||u(\cdot, T_\varepsilon) - U ||_{L^2(\Omega_1)} +  \underset{t\in[T_\varepsilon, \infty]}{\max}||v(\cdot, t) - V ||_{L^2(\Omega_1)} + ||U - V ||_{L^2(\Omega_1)} \\
        \le & C \varepsilon.\\
   \end{split}
\end{equation}
Thus, for a small $\delta > 0$, it holds that
\begin{equation}\label{fkpp: eq424}
    \begin{split}
        \underset{0 \le t \le \infty}{\max}||u(\cdot, t) - v(\cdot, t) ||_{L^2(\Omega_1)} &\le \underset{t\in[0, T_\varepsilon]}{\max}||u(\cdot, t) - v(\cdot, t) ||_{L^2(\Omega_1)}+ \underset{t\in[T_\varepsilon, \infty]}{\max}||u(\cdot, t) - v(\cdot, t) ||_{L^2(\Omega_1)}\\
        & \le C \varepsilon,
    \end{split}
\end{equation}
which completes the proof of the first case.

\medskip 
\noindent
\textbf{Case 2.} $\frac{\sigma}{\delta} \to \alpha \in (0, \infty)$ as $\delta \to 0$. 

Thanks to Theorem \ref{fkpp: thm3} and Lemma \ref{fkpp: le42}, we have $\lambda_1> 0$, where $\lambda_1$ is the principal eigenvalue of \eqref{fkpp: ep2} with the corresponding boundary conditions (third column in Table \ref{fkpp: tb}). 

If $\lambda_1 < 1$, then it holds that $\lambda_1^\delta < 1$ for small $\delta > 0$, from which both $U$ and $V$ exist. Similar to the case as in the proof of Case 1, we complete the proof of this theorem by deriving \eqref{fkpp: eq424}.

If $\lambda_1 > 1$, then it holds that $\lambda_1^\delta > 1$ for small $\delta > 0$, from which neither $U$ nor $V$ exists, namely, $U= V =0$. For small $\delta > 0$, notice
\begin{equation} \label{fkpp: eq425}
   \begin{split}
        \frac{1}{2} \frac{d}{dt}\int_{\Omega}u^2dx =&-\int_{\Omega} \nabla u \cdot A \nabla u + u^2 dx - \int_{\Omega}u^3  dx\\
        \le & (-\lambda_1^\delta + 1) \int_{\Omega} u^2 dx\\
        \le & 0,
   \end{split}
\end{equation}
where we have used \eqref{fkpp: eq42} and $u \ge 0$. Thus, similar to the case as in the proof of Case 1, we get \eqref{fkpp: eq424} that finishes this proof.

\smallskip

If $\lambda_1 = 1$, then we have two possible cases for the behavior of $\lambda_1^\delta$. Suppose after passing to a subsequence of $\delta > 0$, we have $\lambda_1^\delta < 1$ and $\lambda_1^\delta \to \lambda_1$, in which case, $U$ exists along this subsequence but $V$ does not (i.e., $V\equiv 0$). Thus, the proof can be accomplished by the same argument as in Case 1.

On the other hand,  there exists a subsequence of $\delta$, along which we have $\lambda_1^\delta \ge 1$. As a result, neither $U$ nor $V$ exist along this subsequence, namely, $U= V =0$. Moreover, direct computation yields \eqref{fkpp: eq425}, and similar to the case as in the proof of Case 1, we prove this theorem by obtaining \eqref{fkpp: eq424}.
   
\medskip 
\noindent
\textbf{Case 3.} $\frac{\sigma}{\delta} \to \infty$ as $\delta \to 0$. 

If $\lambda_1^D < 1$, it follows from Theorem \ref{fkpp: thm3} that $\lambda_1^\delta < 1$ for small $\delta > 0$, resulting in the existence of $U$ and $V$. By using the method analogous to Case 1, the proof
can be finished.

If $\lambda_1^D > 1$, we have $\lambda_1^\delta > 1$ for small $\delta > 0$, from which $U = V \equiv 0$. By \eqref{fkpp: eq425}, we can prove this theorem by performing the same argument as in Case 1.

If $\lambda_1^D = 1$, we have $\lambda_1^\delta < 1$ for small $\delta > 0$, leading to $U = V \equiv 0$. The proof is same as in Case 1.

\smallskip

Therefore, we complete the whole proof of this theorem.

\end{document}